\documentclass{amsart}
\usepackage[alphabetic, msc-links, backrefs]{amsrefs}
\usepackage{amssymb}
\usepackage{amsthm}
\usepackage{apptools}
\usepackage{chngcntr}
\usepackage{color}
\usepackage{enumerate}
\usepackage{graphicx}
\usepackage{hyperref}
\usepackage{verbatim}
\usepackage[dvipsnames]{xcolor}
\numberwithin{equation}{section}

\newtheorem{theorem}            {Theorem}       [section]      
\newtheorem{lemma}              [theorem]       {Lemma}
\newtheorem{corollary}          [theorem]       {Corollary}
\newtheorem{proposition}        [theorem]       {Proposition}
\newtheorem{example}            [theorem]       {Example}

\newtheorem*{theorem*}                          {Theorem}
\theoremstyle{definition}

\theoremstyle{definition}
\newtheorem{remark}             [theorem]       {Remark}

\allowdisplaybreaks[1]   

\title[Rigidity and Non-existence Results of $\lambda$-Translating Solitons]{Rigidity and Non-existence Results of $\lambda$-Translating Solitons}

\author{Xiang Li$^\ast$}
\author{Jun Sun}

\thanks{
2020 MR Subject Classification: 53C24, 53C40\\
The first author was supported by the National Natural Science Foundation of China (Grant No. 12601094).
The second author was supported by the National Natural Science Foundation of China (Grant No. 12671074, 12531002, 12271039).}	
\thanks{$\ast$ Corresponding author.}

\address{$^1$School of Mathematical Sciences, Jiangsu University, Zhenjiang 212013, China}
\email{lihsiang@mail.ustc.edu.cn}

\address{$^2$School of Mathematics and Statistics, Wuhan University, Wuhan 430072, China}
\email{sunjun@whu.edu.cn}

\begin{document}

\begin{abstract}
A $\lambda$-translating soliton is a hypersurface in $\mathbb{R}^{n+1}$
satisfying $H=\langle \mathbf{T},\nu\rangle+\lambda$; equivalently, it has constant
weighted mean curvature with respect to the log-linear density
$e^{\langle T,X\rangle}$, and is an eternal solution of the mean curvature
flow with a constant forcing term. In this paper, we first prove that every
complete properly immersed $\lambda$-translating soliton with $\lambda>0$ and
$\inf_{\Sigma}H>\lambda$ has at least exponential volume growth, in contrast
with the linear growth of ordinary translating solitons. 
We then prove sharp non-existence results for graphic $\lambda$-translating solitons ($\lambda\geqslant 0$) with bounded gradient, and two rigidity theorems.

\end{abstract}

\noindent

\maketitle

\section{Introduction}

\vspace{.1in}

Fix a constant vector \( \mathbf{T} \neq 0 \) in \( \mathbb{R}^{n+1} \) and a real number \( \lambda \). 
A surface \( \Sigma^n \subset \mathbb{R}^{n+1} \) is said to be a \emph{\( \lambda \)-translating soliton} if it is a critical point of the weighted area functional
\[
\mathcal{A} \triangleq \int_{\Sigma} e^{\langle \mathbf{T},X\rangle}\,d\mu
\]
with given weighted volume for the domain  enclosed by $\Sigma$. Here $X$ is the position vector of ${\mathbb R}^{n+1}$, and $d\mu$ is the induced area form on $\Sigma$.

A \( \lambda \)-translating soliton satisfies the equation  
\begin{equation}\label{eq-lambdatranslatingsoliton}
	H = \langle \mathbf{T}, \nu \rangle + \lambda,
\end{equation}
where \( \nu \) is the unit inner normal vector of \( \Sigma \), \( H \) is the mean curvature of \( \Sigma \), and we denote by \( \langle \cdot, \cdot \rangle \) the standard inner product on \( \mathbb{R}^{n+1} \).  
Without loss of generality, in this paper we always assume that $|\mathbf{T}| = 1$.

The variational definition above already contains a precise geometric
meaning of the parameter $\lambda$. Endow $\mathbb{R}^{n+1}$ with the
log-linear density $e^{\langle \mathbf{T},X\rangle}$, that is, measure volumes
and perimeters with respect to $e^{\langle \mathbf{T},X\rangle}\,dv$ and
$e^{\langle \mathbf{T},X\rangle}\,d\mu$. Such \emph{manifolds with density} are a
standard setting for the weighted isoperimetric problem (\cite{morgan2005manifolds},
\cite{rosales2008isoperimetric}). The first variation formula in this setting shows that
$\Sigma$ is a critical point of the weighted area under the
weighted-volume constraint if and only if its \emph{weighted mean
	curvature}
\[
H_{\phi}\triangleq H-\langle \mathbf{T},\nu\rangle
\]
is constant, and the Lagrange multiplier of the volume constraint is
exactly $\lambda$. In other words, $\lambda$-translating solitons are
precisely the surfaces of constant weighted mean curvature in
$(\mathbb{R}^{n+1},e^{\langle T,X\rangle}dv)$, in the same way that
constant mean curvature surfaces are the critical points of the ordinary
area functional under a volume constraint. Consequently,
$\lambda$-translating solitons are the natural candidates for the
boundaries of isoperimetric regions in the weighted isoperimetric
problem, and the
parameter $\lambda$ plays the role of a ``weighted pressure'': it is the
constant value of the weighted mean curvature enforced by the volume
constraint, the borderline case $\lambda=0$ being the weighted
area-stationary surfaces.

The class also plays a direct role in geometric evolution, which goes
beyond the variational motivation. Rearranging equation \eqref{eq-lambdatranslatingsoliton} as
\[
\langle \mathbf{T},\nu\rangle=H-\lambda,
\]
we see that $\Sigma$ is a $\lambda$-translating soliton if and only if
the family $\Sigma+t\mathbf{T}$ evolves by the \emph{forced mean curvature flow}
\[
\frac{\partial X}{\partial t}=(H-\lambda)\nu,
\]
that is, the mean curvature flow with a constant forcing term in the
normal direction. 
Such forced flows have been studied in both Euclidean and Minkowski spaces (\cite{liu2007evolution}, \cite{aarons2006mean}); in particular, for a constant forcing term, in the Minkowski space setting, Aarons \cite{aarons2006mean} proved that solutions of the forced flow converge, depending on the behaviour at infinity, either to a constant mean curvature hypersurface or to a translating solution of the forced flow, that is, precisely to a $\lambda$-translating soliton in the Lorentzian sense.
Hence a $\lambda$-translating soliton is an eternal
translating solution of this forced flow. When
$\lambda=0$, the forced flow reduces to the mean curvature flow itself;
in this case $\lambda$-translating solitons are also called translating
solitons, which are important models of the Type II singularities of the
mean curvature flow (\cite{hamilton1995harnack}, \cite{huisken1999mean}). When $\lambda\neq0$,
the additional term $-\lambda\nu$ represents a constant external driving
force on the evolving interface, and looking for translating solutions
of such a forced flow leads precisely to equation \eqref{eq-lambdatranslatingsoliton}. 
We should point out that the Type II singularity interpretation mentioned above belongs to the case $\lambda=0$: for $\lambda\neq0$ it is not known whether $\lambda$-translating solitons occur as blow-up models of the forced flow. Our motivation for the general $\lambda$-dependent class is instead twofold: the weighted isoperimetric problem described above, and the forced flow picture, in which the parameter $\lambda$ captures the strength and the sign of the constant forcing on the evolving interface.

Recently, López (\cite{lopez2018invariant}) classified  all \( \lambda \)-translating solitons in \( \mathbb{R}^3 \) that are invariant by a one-parameter group of translations and a one-parameter group of rotations. 
In \cite{lopez2018compact}, he studied the shape of a compact \( \lambda \)-translating soliton of \( \mathbb{R}^3 \) in terms of the geometry of its boundary, obtaining some necessary conditions for the existence of two-dimensional compact \( \lambda \)-translating solitons with a given closed boundary curve. 
In particular, he proved that there does not exist any closed \( \lambda \)-translating soliton of dimension two (actually, his argument also holds for any dimensional case).  

\vspace{.1in}

In the first part of the paper, we consider the volume growth of $\lambda$-translating solitons. 
Understanding the volume growth of certain geometric solitons is always a fundamental and interesting topic in the study of geometric flows, such as the volume growth of translating solitons for mean curvature flow. 
For instance, Guang (\cite{guang2019volume}) proved that every complete properly immersed translating soliton has at least linear volume growth.
Nguyen \cite{nguyen2015doubly} proved that translating solitons do not necessarily have Euclidean volume growth.
For $\lambda$-translating solitons, we prove the following result concerning its volume growth property:

\begin{theorem}
	Let $\Sigma^n \subset \mathbb{R}^{n+1}$ be a complete properly immersed $\lambda$-translating soliton with $\lambda>0$ and $H_{\inf}>\lambda,$ where $\inf_{\Sigma}H:\triangleq H_{\inf}$. Then for any $x \in \Sigma$, there exists a constant $C$ such that
	\[
	\operatorname{Vol}(\Sigma \cap B_r(x)) \geqslant \frac{\tilde{V}(1)}{e^{\big(1+\lambda(H_{\inf}-\lambda)\big)}}\bigl(\frac{1+\lambda(H_{\inf}-\lambda)}{\lambda(H_{\inf}-\lambda)}\bigr)e^{\lambda(H_{\inf}-\lambda)r}+C, 
	\]
    for all $r \geqslant 1$, where $\tilde{V}(r)$ is the weighted area defined by (\ref{e-weighted-area}) and 
    \begin{equation*}
        C=V(1)-\frac{\tilde{V}(1)}{e}\left(\frac{1+\lambda(H_{\inf}-\lambda)}{\lambda(H_{\inf}-\lambda)}\right).
    \end{equation*}
\end{theorem}

\begin{remark}
	Checking the proof of the above theorem in Section 3 carefully, we see that when $\lambda=0$, every complete properly immersed translating soliton has at least linear volume growth, which is the result obtained by Guang (\cite{guang2019volume}). However, when $\lambda>0$ and $H_{\inf}-\lambda>0$, the complete properly immersed $\lambda$-translating soliton has at least exponential volume growth.
\end{remark}

\vspace{.1in}

In the second part of the paper, we can give a non-existence result of graphic
$\lambda$-translating solitons in $\mathbb{R}^3$ and a non-existence result of rotationally symmetric graphic $\lambda$-translating solitons in $\mathbb{R}^{n+1}$.  

By definition, after a translation and rotation, any $\lambda$-translating soliton $\Sigma^n$ in $\mathbb{R}^{n+1}$ with positive $H-\lambda$ can be represented as a graph of some function $u$. 
When a $\lambda$-translating soliton $\Sigma^n$ in $\mathbb{R}^{n+1}$ can be represented as a graph,
it is easy to see that \eqref{eq-lambdatranslatingsoliton} can be written as
\begin{equation}\label{eq-graphlambdatranslatingsoliton}
	\operatorname{div}\left(\frac{Du}{\sqrt{1+|Du|^2}}\right)=\frac{1}{\sqrt{1+|Du|^2}}+\lambda.
\end{equation}
In this case, $\mathbf{T}=(0,\dots,0,1)$. We first prove that

\begin{theorem}\label{Theoremnoentire1}
	When $n = 2$ and $\lambda\geqslant0$, there is no entire solution to the equation \eqref{eq-graphlambdatranslatingsoliton}
	with bounded gradient.
\end{theorem}

\begin{remark}
    The assumption \(\lambda \geqslant 0\) in Theorem \ref{Theoremnoentire1} is necessary: when \(\lambda \in (-1,0)\), the affine function \(u(x)=\overrightarrow{a}\cdot \overrightarrow{x}+b\) (taking \(|\overrightarrow{a}|=\sqrt{1/\lambda^2-1}\)) is a bounded gradient entire solution of \eqref{eq-graphlambdatranslatingsoliton}, where $\overrightarrow{a},\overrightarrow{x}\in\mathbb{R}^n$ and $b\in\mathbb{R}$.
\end{remark}

Furthermore, when a $\lambda$-translating soliton $\Sigma^n$ in $\mathbb{R}^{n+1}$ is a radial graph of some function $u=u(r)$ , \eqref{eq-graphlambdatranslatingsoliton} can be written as 
\begin{equation}\label{eq-rotationalgraphlambdatranslatingsoliton}
    \frac{u''}{\bigr(1+(u')^{2} \bigl)^{\frac{3}{2}}}+\frac{(n-1)u'}{r\sqrt{1+(u')^2}}=\frac{1}{\sqrt{1+(u')^{2} }}+\lambda.
\end{equation}
We can show that

\begin{theorem}\label{Theoremnoentire2}
    There does not exist a rotationally symmetric graphical complete $\lambda$-translating soliton
    with bounded gradient for $\lambda\geqslant0$.
\end{theorem}

\begin{remark}
    The assumption \(\lambda \geqslant 0\) in Theorem \ref{Theoremnoentire2} is necessary. For every \(\lambda \in (-1,0)\), equation \eqref{eq-rotationalgraphlambdatranslatingsoliton} admits a global solution \(u \in C^{\infty}[0,\infty)\) with \(u(0)=u'(0)=0\), satisfying \(0 < u' < m_{\lambda} := \sqrt{\lambda^{-2}-1}\) and \(u(r) = m_{\lambda} r - \frac{n-1}{\lambda^2} \log r + O(1)\) as \(r \to \infty\). Indeed, writing \(v = u'\), one has \(v' = (1+v^2)\left[1 + \lambda\sqrt{1+v^2} - \frac{(n-1)v}{r}\right]\); the sign of the right-hand side at \(v=0\) and \(v=m_{\lambda}\) forces \(0 < v < m_{\lambda}\) for all \(r > 0\). These are strictly convex entire rotationally symmetric graphs with bounded gradient, asymptotic to a cone of slope \(m_{\lambda}\); in \(\mathbb{R}^3\) they are exactly the rotational \(\lambda\)-translators classified by López\cite{lopez2018invariant}.
\end{remark}

\vspace{.1in}

As is known, there have been many rigidity theorems and classification theorems for translating solitons in the Euclidean space and the pseudo-Euclidean space. 
For example, Chen--Qiu (\cite{chen2016rigidity}) proved that there exists no complete \( m \)-dimensional spacelike translating soliton in \( \mathbb{R}^{m+n}_{n} \). 
The classification theorem also exists for \( \lambda \)-translating solitons. 
For example, Li--Qiao--Liu (\cite{li2020complete}) have classified
complete \( \lambda \)-translating solitons in the Euclidean space \( \mathbb{R}^3 \) and the Minkowski space \( \mathbb{R}^3_1 \) with second fundamental form of constant length. 
For the higher dimension \( n \), Li--Wei (\cite{li2023complete}) also obtain a similar classification for 3-dimensional complete \( \lambda \)-translating solitons in \( \mathbb{R}^4_1 \).
In the third part of the paper, we will give two rigidity results of
$\lambda$-translating solitons in $\mathbb{R}^{n+1}$.

For a submanifold in Euclidean space we can define its Gauss map. 
By studying this mapping, we can get a lot of information about submanifolds. Let \( \gamma_1 : \Sigma^n \to {\mathbb S}^n \) be the Gauss map of the translating soliton \( \Sigma^n \) (with \( \lambda = 0 \)), and the image \( \gamma_1(x) \) be the unit outer normal vector of \( \Sigma^n \).
Here, and only in the statement of Theorem 1.7, we use the outer unit
normal as in \cite{bao2014gauss}; since replacing $\nu$ by $-\nu$ reverses the signs of
$H$ and $B$ simultaneously and maps the Gauss image to its antipodal set,
this convention does not affect the statement. In the rest of the paper
$\nu$ always denotes the unit inner normal.
Bao--Shi (\cite{bao2014gauss}) proved the following Bernstein theorem for translating solitons, which can be seen as a generalization of classical Bernstein theorem for entire minimal graphs:

\begin{theorem}[\cite{bao2014gauss}]\label{baoshitheo}
	Let \( \Sigma^n \subset \mathbb{R}^{n+1} \) be an \( n \)-dimensional complete translating soliton with bounded mean curvature. 
	If the image of Gauss map \( \gamma_1 \) of \( \Sigma^n \) lies in a ball \( B_\Lambda^S(y_0) \) of \( S^n \), where \( \Lambda < \frac{\pi}{2} \), then \( \Sigma^n \) must be a hyperplane. 
	Here and in the sequel \( B_\Lambda^S(y_0) \) denotes the geodesic ball of radius \( \Lambda \) and with center \( y_0 \) in \( S^n \).
\end{theorem}

On the other hand, Xin (\cite{xin2015translating}) proved several properties for translating solitons with arbitrary codimension. 
One of his Bernstein theorems is as follows:

\begin{theorem}[\cite{xin2015translating}]\label{thm-TSunderInt}
	Let \( \Sigma^n \) be a complete immersed translating soliton in \( \mathbb{R}^{n+1} \). 
	If \( \Sigma^n \) satisfies the integral conditions
	\[
	\left( \int_\Sigma |B|^n \right)^{\frac{1}{n}} < \sqrt{\frac{2(n-1)}{n^2\kappa}} \quad \text{and} \quad \int_\Sigma |B|^n e^{\langle \mathbf{T},X\rangle} < \infty,
	\]
	where \( \kappa \) is the Sobolev constant in Lemma \ref{lem:sobolev} and $B$ is the second fundamental form of $\Sigma^n$ in ${\mathbb R}^{n+1}$, then \( |B| \equiv 0 \) and \( \Sigma \) is a hyperplane.
\end{theorem}

In this paper, we will generalize the above two theorems to the case that \( \Sigma \) is a \( \lambda \)-translating soliton. 
More precisely, for the Gauss map of $\lambda$-translating solitons $\gamma:\Sigma\to\mathbb{S}^n$, we first prove that

\begin{theorem}\label{thm-GaussMapoflambdatranslators}
	Let $\Sigma^n \subset \mathbb{R}^{n+1}$ be an $n$-dimensional complete $\lambda$-translating soliton with $\lambda B$ positive semi-definite. If the image of the Gauss map $\gamma$ of $\Sigma^n$ lies in a ball $B_\Lambda^{S^n}(y_0)$ of $\mathbb{S}^n$, where $\Lambda < \frac{\pi}{2}$, then $\Sigma^n$ must be a hyperplane.
\end{theorem}

\begin{remark}
Unlike Theorem \ref{baoshitheo}, Theorem \ref{thm-GaussMapoflambdatranslators} requires no boundedness assumption on the
mean curvature, since $|H|\leqslant1+|\lambda|$ is a direct consequence of the
soliton equation \eqref{eq-lambdatranslatingsoliton}. 
\end{remark}

Here we present two examples that do not satisfy the Gauss image condition of Theorem \ref{thm-GaussMapoflambdatranslators} (\cite{lopez2018invariant}).

\begin{example}
	A circular cylinder of radius $r>0$ whose axis is parallel to ${\bf T}$ is a $1/r$-translating soliton.
\end{example}

\begin{example}
	The generating curve $\alpha(s) = \bigl(y(s), z(s)\bigr)$ of a cylindrical $\lambda$-translating soliton $X(s,t) = \bigl(t, y(s), z(s)\bigr)$, $s \in \mathbb{R}$, $t \in \mathbb{R}$ is:
	\[
	y(s) = -\lambda s + 2\arctan\left(
	\frac{\lambda + 1}{\sqrt{1 - \lambda^2}}
	\tanh\left(\frac{\sqrt{1 - \lambda^2}}{2}s\right)
	\right),
	\]
	\[
	z(s) = \log\left|-\lambda + \cosh\bigl(\sqrt{1 - \lambda^2}s\bigr)\right|,
	\]
	where $0 < \lambda < 1$.
\end{example}

Inspired by Theorem \ref{thm-TSunderInt}, we derive integral estimates for the squared norm of the second fundamental form \( |B|^2 \) in terms of the Sobolev constant by using the Sobolev inequality and the Bochner type inequality for \( \lambda \)-translating soliton. 
This leads to the following rigidity theorem:

\begin{theorem}
	Let \( \Sigma^n \) be a complete immersed \( \lambda \)-translating soliton in \( \mathbb{R}^{n+1} \). 
	If \( \lambda H\leqslant 0 \), and \( \Sigma^n \) satisfies integral conditions
	\[
	\left( \int_\Sigma |B|^n \right)^{\frac{2}{n}} + |\lambda| \left( \int_\Sigma |B|^{\frac{n}{2}} \right)^{\frac{2}{n}} < \frac{2(n-1)}{n^2\kappa} \quad \text{and} \quad \int_\Sigma |B|^n e^{\langle \mathbf{T},X\rangle} < \infty,
	\]
	where \( \kappa \) is the Sobolev constant, then \( |B| \equiv 0 \) and \( \Sigma \) is a hyperplane.
\end{theorem}

\begin{remark}
	In this theorem, we only consider the rigidity theorem for \( \lambda \)-translating soliton with \( \lambda H\leqslant 0 \). 
\end{remark}

The subsequent sections are organized as follows: in Section 2, we present fundamental materials for the $\lambda$-translating solitons; in Section 3, we provide the lower bound for the area of the extrinsic balls; in Section 4, we prove two non-existence results for the graphical $\lambda$-translating solitons; in the last section, we will derive the rigidity theorems.

\vspace{.1in}

\textbf{Conflict of Interest} The authors have no conflict of interest to declare.

\vspace{.1in}
\textbf{Data availability} The authors declare no datasets were generated or analysed during the current study.

\vspace{.2in}

\section{Preliminaries}

\vspace{.1in}

Let $\nabla$ and $\overline{\nabla}$ be the Levi-Civita connections on $\Sigma^{n}$ and $\mathbb{R}^{n+1}$, respectively. 
Then we define the second fundamental form $B$ by
\[
B(V,W) = (\overline{\nabla}_V W)^{\bot} = \overline{\nabla}_V W - \nabla_V W
\]
for any tangent vector fields $V, W$ in $\Sigma^{n}$, where $(\cdots)^{\bot}$ stands for the orthogonal projection into the normal bundle $N\Sigma^{n}$. 
Similarly, $(\cdots)^{\top}$ stands for the tangential projection. 
The mean curvature vector $\mathbf{H}$ of $\Sigma^n$ is given by
\[
\mathbf{H} = \operatorname{trace}(B) = \sum_{i=1}^n B(e_i,e_i) \in \Gamma(N\Sigma^{n}),
\]
where $\{e_i\}$ is a local orthonormal frame field of $\Sigma^{n}$. 
In what follows we use $\nabla$ for natural connections on various bundles for notational simplicity if there is no ambiguity from the context. 
For $\nu \in \Gamma(N\Sigma^{n})$ the shape operator $A^\nu \colon T\Sigma^{n} \to T\Sigma^{n}$, defined by
\[
A^\nu(V) = -(\overline{\nabla}_V \nu)^{\top}=-\overline{\nabla}_V \nu,
\]
satisfies
\[
\langle B_{VW}, \nu \rangle = \langle A^\nu(V), W \rangle.
\]
Set
\begin{equation*}
	B_{e_i e_j}=h_{ij}\nu,
\end{equation*}
then 
\begin{equation*}
	H=g^{ij}h_{ij},
\end{equation*}
and
\begin{equation*}
	|B|^2=g^{ij}g^{kl}h_{ik}h_{jl}.
\end{equation*}

For the sake of convenience, we sometimes use $\mathbf{V} = \mathbf{T}^{\top}$ to denote the tangential part of $\mathbf{T}$.
\begin{lemma}\label{lem:rho}
	Let $\Sigma^n \subset \mathbb{R}^{n+1}$ be a $\lambda$-translating soliton and denote $\rho = e^{\langle \mathbf{T},X \rangle}$, where $X$ is position vector in ${\mathbb R}^{n+1}$. Then, we have
	\begin{equation}\label{eq:grad-rho}
		\nabla\rho = \rho \mathbf{V},
	\end{equation}
	\begin{equation}\label{eq:grad-rho-half}
		\nabla\rho^{\frac{1}{2}} = \frac{1}{2}\rho^{\frac{1}{2}}\mathbf{V}
	\end{equation}
    and
	\begin{equation}\label{eq:laplace-rho}
		\Delta\rho = \rho\big(1 + \lambda(H - \lambda)\big).
	\end{equation}
\end{lemma}

\begin{proof}
	From \eqref{eq-lambdatranslatingsoliton}, we have
	\[
	1 = |\mathbf{T}|^2 = |\mathbf{V}|^2 + (H - \lambda)^2,
	\]
	\[
	|\mathbf{V}|^2 + H(H - \lambda) = 1 - (H - \lambda)^2 + H(H - \lambda) = 1 + \lambda(H - \lambda),
	\]
	where $\mathbf{V}$ is the tangential part of $\mathbf{T}$. Then, we have
	\[
	\nabla\rho = \rho\nabla\langle \mathbf{T},X \rangle = \rho\langle \mathbf{T},e_i \rangle e_i = \rho \mathbf{T}^\top = \rho \mathbf{V},
	\]
	\[
	\nabla\rho^{\frac{1}{2}} = \frac{1}{2}\rho^{\frac{1}{2}}\mathbf{V},
	\]
	\[
	\begin{aligned}
		\Delta\rho &= \nabla_{e_i}\nabla_{e_i}\rho = \nabla_{e_i}\big(\rho\langle \mathbf{T},e_i \rangle\big) = \rho\langle \mathbf{T},e_i \rangle\langle \mathbf{T},e_i \rangle + \rho\langle \mathbf{T},\mathbf{H} \rangle \\
		&= \rho\big(|\mathbf{V}|^2 + H(H - \lambda)\big) = \rho\big(1 + \lambda(H - \lambda)\big).
	\end{aligned}
	\]
\end{proof}

Next we will compute the general formulas:

\begin{proposition}
	Let \( \Sigma^{n} \) be an oriented immersed hypersurface in \( \mathbb{R}^{n+1} \) with second fundamental form \( B \).  
	Then for any fixed vector \( \mathbf{W} \in \mathbb{R}^{n+1} \),
	\begin{equation*}
		\begin{split}
			\nabla_{e_i} \langle \mathbf{W}, \nu \rangle = -\langle \mathbf{W}, e_k\rangle h_{ik},
		\end{split}
	\end{equation*}
	and
	\begin{equation*}
		\begin{split}
			\nabla_{e_i}\nabla_{e_j} \langle \mathbf{W}, \nu \rangle =  -\langle \mathbf{W}, \nabla h_{ij} \rangle -\langle \mathbf{W}, \nu \rangle h_{ik}h_{kj},
		\end{split}
	\end{equation*}
	where  \( \nu \) is the unit inner normal vector to \( \Sigma^{n} \).
\end{proposition}

\begin{proof}
	Choose a local orthonormal frame field \(\{e_i\}\) along $\Sigma$ with \(\nabla_{e_j} e_i = 0\) at the considered point, so that at the point we have
    \begin{equation*}
        \overline{\nabla}_{e_i}e_j=h_{ij}\nu, \ \ \overline{\nabla}_{e_i} \nu=-h_{ij}e_j. 
    \end{equation*}
    Then
	\begin{equation*}
		\begin{split}
			\nabla_{e_i} \langle \mathbf{W}, \nu \rangle &= \langle \mathbf{W}, \overline{\nabla}_{e_i} \nu \rangle = -\langle \mathbf{W}, e_k\rangle h_{ik},
		\end{split}
	\end{equation*}
	and
	\begin{equation*}
		\begin{split}
			\nabla_{e_i}\nabla_{e_j} \langle \mathbf{W}, \nu \rangle
			&= \nabla_{e_i} \langle \mathbf{W}, \overline{\nabla}_{e_j} \nu \rangle \\
			&= \langle \mathbf{W}, \overline{\nabla}_{e_i} \overline{\nabla}_{e_j} \nu \rangle \\
			&= -\langle \mathbf{W}, \overline{\nabla}_{e_i} A^\nu (e_j) \rangle \\
			&= -\langle \mathbf{W}, \nabla_{e_i} A^\nu (e_j)+B(e_i,A^\nu (e_j)) \rangle \\
			&= -\langle \mathbf{W}, \nabla_{e_i} A^\nu (e_j)+B(e_i,h_{jk}e_k) \rangle \\
			&= -\langle \mathbf{W}, \nabla_{e_i} A^\nu (e_j)+h_{ik}h_{kj}\nu \rangle.
		\end{split}
	\end{equation*}
On the other hand, we have from $A^\nu (e_j)=h_{jk}e_k$ that
\begin{equation*}
		\begin{split}
			\nabla_{e_i} A^\nu (e_j)
			&= \nabla_{e_i}(h_{jk}e_k)
            =(\nabla_{e_i}h_{jk})e_k 
            =(\nabla_{e_k}h_{ij})e_k = \nabla h_{ij}.
		\end{split}
	\end{equation*}
    Here, we used the Codazzi equation in the third equality.	Therefore, we have
	\begin{equation*}
		\begin{split}
			\nabla_{e_i}\nabla_{e_j} \langle \mathbf{W}, \nu \rangle 
			&= -\langle \mathbf{W}, \nabla h_{ij} \rangle -\langle \mathbf{W}, \nu \rangle h_{ik}h_{kj}.
		\end{split}
	\end{equation*}
\end{proof}

\vspace{.1in}

As a consequence, we have

\begin{corollary}
	Let \( \Sigma^{n} \) be an oriented $\lambda$-translating soliton in \( \mathbb{R}^{n+1} \) with second fundamental form \( B \).  
	Then
	\begin{equation}\label{eq:second-deriv-H}
		\begin{split}
			\nabla_i \nabla_j H =& - \langle {\bf T}, \nu \rangle h_{ik}h_{jk} - \langle {\bf T}, \nabla h_{ij}\rangle \\
			=&-(H-\lambda)h_{ik}h_{jk} - \langle {\bf T}, e_k \rangle \nabla_k h_{ij},
		\end{split}
	\end{equation}
	and
	\begin{equation}\label{eq-Deltaanu}
		\Delta \langle {\bf T}, \nu \rangle + |B|^2 \langle {\bf T}, \nu \rangle + \langle {\bf T}, \nabla H \rangle = 0.
	\end{equation}
\end{corollary}

According to Simons' identity and \eqref{eq:second-deriv-H}, for $\lambda$-translating solitons in $\mathbb{R}^{n+1}$ we can derive the Laplacian of the squared norm of the second fundamental form as follows:
\[
\begin{aligned}
	\Delta h_{ij} &= \nabla_i \nabla_j H + H h_{ik} h_{kj} - |B|^2 h_{ij} \\
	&= \langle \lambda - H, h_{ik} \rangle h_{jk} - \langle T, e_k \rangle \nabla_k h_{ij} + H h_{ik} h_{kj} - |B|^2 h_{ij} \\
	&= -\langle T, e_k \rangle \nabla_k h_{ij} - |B|^2 h_{ij} + \lambda h_{ik} h_{jk}.
\end{aligned}
\]
Next, we introduce a linear operator on $\Sigma^n$
\[
\mathcal{L} = \Delta + \langle \mathbf{T},\nabla(\cdot) \rangle = e^{-\langle \mathbf{T},X \rangle}\operatorname{div}\big(e^{\langle \mathbf{T},X \rangle}\nabla(\cdot)\big),
\]
in a similar manner of the drift-Laplacian on the self-shrinkers by Colding and Minicozzi in \cite{colding2012generic}. It can be shown that $\mathcal{L}$ is self-adjoint with respect to the measure $e^{\langle \mathbf{T},X \rangle}d\mu$.

As a consequence, we have
\begin{equation}\label{eq:L-A-squared}
	\mathcal{L}|B|^2 = \Delta|B|^2 + \langle \mathbf{T}, \nabla|B|^2 \rangle = 2|\nabla B|^2 - 2|B|^4 + 2\lambda \operatorname{tr}(B^3),
\end{equation}
and
\[
\mathcal{L}H = \Delta H + \langle \mathbf{T}, \nabla H \rangle = (\lambda - H)|B|^2.
\]

\vspace{.1in}

Recall that we have
\begin{equation}\label{e-T}
    {\bf T}={\bf V}+(H-\lambda)\nu=V^je_j+(H-\lambda)\nu.
\end{equation}
Take covariant derivative on $V$ and equating tangential and normal components
we find that
\begin{equation}\label{eq:main}
	\begin{cases}
		\nabla_i V^j = (H - \lambda) h_{ik} g^{kj} \\
		\nabla_i H + h_{ij} V^j = 0.
	\end{cases}
\end{equation}

\vspace{.1in}

Next, we consider the Gauss map of $\lambda$-translating solitons $\gamma:\Sigma\to\mathbb{S}^n$.
The pullback under $\gamma$ of the tangent bundle $T\mathbb{S}^n$ to a bundle over $\Sigma$ is denoted by $\gamma^{-1}T\mathbb{S}^n$.
Denote by $\Gamma$ and $\tilde{\Gamma}$ the Christoffel symbols on $\Sigma$ and $\mathbb{S}^n$, and by $\{x^i\}$ and $\{y^\alpha\}$ the local coordinates on $\Sigma$ and $\mathbb{S}^n$ respectively.
Similar to Bao-Shi's analysis in \cite{bao2014gauss}, we have the following lemma:

\begin{lemma}\label{lem:quasi-harmonic}
	The Gauss map \( \gamma \) of a \( \lambda \)-translating soliton \( \Sigma \) forms a quasi-harmonic map, i.e.
	\[
	\tau(\gamma) = -\nabla {\bf H} = - d\gamma({\bf V})
	\]
	where 
	\[
	\tau^\alpha(\gamma) = \Delta \gamma^\alpha + g^{ij}\tilde{\Gamma}^\alpha_{\beta\sigma}\frac{\partial \gamma^\beta}{\partial x^i}\frac{\partial \gamma^\sigma}{\partial x^j}.
	\]
\end{lemma}

\begin{lemma}\label{lem:bochner}
	For the quasi-harmonic equation of a Gauss map $\gamma$ of a $\lambda$-translating soliton, we have the following Bochner formula:
	\begin{equation*}\label{eq:bochner}
		\Delta|\nabla \gamma|^2 = 2|\nabla d\gamma|^2 - 2|\nabla \gamma|^4 - \langle {\bf V}, \nabla|\nabla\gamma|^2 \rangle + 2\lambda B(\nabla \gamma, \nabla \gamma).
	\end{equation*}
\end{lemma}
\begin{proof}
Denote $R_{ijkl}$ and $K_{\alpha \beta \iota \sigma}$ be the curvature operator on $\Sigma$ and ${\mathbb S}^n$ both with induced metric respectively. In general, we have the Bochner formula:
\[
\Delta|\nabla \gamma|^2 = 2|\nabla d\gamma|^2 + 2\langle d\tau(\gamma), d\gamma \rangle + 2R_{ij}\langle \gamma_i, \gamma_j \rangle - 2K_{\alpha \beta \iota \sigma} \gamma_i^\alpha \gamma_j^\beta \gamma_i^\iota \gamma_j^\sigma,
\]
By Gauss equation for hypersurface, we have
\[
R_{ij} = H h_{ij} - h_{ik} h_{kj}.
\]
On the other hand, the curvature tensor on ${\mathbb S}^n$ is given by
\[
K_{\alpha \beta \iota \sigma} = \delta_{\alpha \iota} \delta_{\beta \sigma} - \delta_{\alpha \sigma} \delta_{\beta \iota}.
\]
Since $\gamma$ is Gauss map, Weingarten formula gives $\gamma_i = -h_{ij} e_j$. Together with \eqref{eq:main}, we compute
\[
\begin{aligned}
	\langle d\tau(\gamma), d\gamma \rangle &= \langle d\gamma, d(-d\gamma(\mathbf{V})) \rangle = -\langle d\gamma, d\gamma(\nabla \mathbf{V}) \rangle - \langle d\gamma, \nabla_\mathbf{V} d\gamma \rangle \\
	&= -(H - \lambda)B(\nabla \gamma, \nabla \gamma) - \frac{1}{2}\langle \mathbf{V}, \nabla|\nabla \gamma|^2 \rangle.
\end{aligned}
\]
Using Gauss equation, we have
\[
\begin{aligned}
	R_{ij}\langle \gamma_i, \gamma_j \rangle &= H B(\nabla \gamma, \nabla \gamma) - h_{ik} h_{kj} h_{ip} h_{jp} \\
	&= H B(\nabla \gamma, \nabla \gamma) - \sum_{i,j} \left( \sum_k h_{ik} h_{kj} \right)^2,
\end{aligned}
\]
and
\[
\begin{aligned}
	K_{\alpha\beta\iota\sigma} \gamma_i^\alpha \gamma_j^\beta \gamma_i^\iota \gamma_j^\sigma &= \delta_{\alpha\iota}\delta_{\beta\sigma} \gamma_i^\alpha \gamma_j^\beta \gamma_i^\iota \gamma_j^\sigma - \delta_{\alpha\sigma}\delta_{\beta\iota} \gamma_i^\alpha \gamma_j^\beta \gamma_i^\iota \gamma_j^\sigma \\
	&= |\nabla \gamma|^4 - \sum_{i,j} \left( \langle \gamma_i, \gamma_j \rangle \right)^2 \\
	&= |\nabla \gamma|^4 - \sum_{i,j} \left( \sum_k h_{ik} h_{kj} \right)^2.
\end{aligned}
\]
The assertion follows from the above equalities.
\end{proof}

\vspace{.2in}

\section{Volume Growth for \texorpdfstring{$\lambda$}--Translating Solitons}

\vspace{.1in}

In this section, we consider the volume growth for $\lambda$-translating solitons and show that every properly immersed $\lambda$-translating soliton with $\lambda>0$ and $\inf_{\Sigma}\langle \mathbf{T},\nu\rangle=H_{\inf}-\lambda>0$ has at least exponential volume growth.

Suppose $\Sigma^n \subset \mathbb{R}^{n+1}$ is a properly immersed $\lambda$-translating soliton. 
For any $x_0 \in \Sigma$, let $B_r(x_0)$ be the extrinsic ball in $\mathbb{R}^{n+1}$, and denote the volume and the weighted volume of $\Sigma \cap B_r(x_0)$ by
\[
V(r) = \operatorname{Vol}(\Sigma \cap B_r(x_0)) = \int_{\Sigma \cap B_r(x_0)} d\mu,
\]
and
\begin{equation}\label{e-weighted-area}
    \tilde{V}(r) = \int_{\Sigma \cap B_r(x_0)} e^{\langle \mathbf{T},X\rangle} d\mu.
\end{equation}

We will first show that the weighted volume has at least exponential growth.
For simplicity, we may assume $x_0 = 0$. 
By the co-area formula, we have
\[
\tilde{V}(r) = \int_0^r \int_{\partial B_r \cap \Sigma} e^{\langle \mathbf{T},X\rangle} \frac{1}{|\nabla_\Sigma |X| |} \, ds.
\]
Note that $\nabla_\Sigma |X|^2 = 2X^{\top} = 2|X| \nabla_\Sigma |X|$, so we get
\begin{equation}\label{eq-tildeV'}
	\tilde{V}'(r) = \int_{\partial B_r \cap \Sigma} e^{\langle \mathbf{T},X\rangle} \frac{|X|}{|X^{\top}|}.
\end{equation}
On the other hand, by \eqref{eq:laplace-rho}, we obtain that if $1+\lambda(H_{\inf}-\lambda)>0$, then
\begin{equation}\label{inq-tildeV}
	\begin{split}
		\tilde{V}(r) &= \int_{\Sigma \cap B_r} \frac{\Delta e^{\langle \mathbf{T},X\rangle}}{1+\lambda(H-\lambda)}\, d\mu \\
		&\leqslant \frac{1}{1+\lambda(H_{\inf}-\lambda)}\int_{\partial B_r \cap \Sigma} \left\langle \nabla_\Sigma e^{\langle \mathbf{T},X\rangle}, \frac{X^{\top}}{|X^{\top}|} \right\rangle \\
		&= \frac{1}{1+\lambda(H_{\inf}-\lambda)}\int_{\partial B_r \cap \Sigma} \left\langle \nabla_\Sigma \langle \mathbf{T},X\rangle, \frac{X^{\top}}{|X^{\top}|} \right\rangle e^{\langle \mathbf{T},X\rangle}.
	\end{split}
\end{equation}
Combining \eqref{eq-tildeV'} and \eqref{inq-tildeV}, we conclude that for any $r > 0$,
\begin{equation}\label{inq-tildeV'geqtildeV}
	\tilde{V}(r) \leqslant \frac{1}{1+\lambda(H_{\inf}-\lambda)}\tilde{V}'(r).
\end{equation}
This implies the quantity $\tilde{V}(r)e^{-\big(1+\lambda(H_{\inf}-\lambda)\big)r}$ is monotone non-decreasing. We summarize this result in the following proposition.

\begin{proposition}\label{inq-tildeVgeqe}
	Let $\Sigma^n \subset \mathbb{R}^{n+1}$ be a complete properly immersed $\lambda$-translating soliton with $\lambda>0$ and $1+\lambda(H_{\inf}-\lambda)>0$. Then for any $x \in \Sigma$, there exists a constant $C = \tilde{V}(1)e^{-\big(1+\lambda(H_{\inf}-\lambda)\big)} > 0$ such that
	\[
	\tilde{V}(r) = \int_{\Sigma \cap B_r(x)} e^{\langle \mathbf{T},X\rangle} d\mu \geqslant C e^{\big(1+\lambda(H_{\inf}-\lambda)\big)r}, \quad \text{for all } r \geqslant 1.
	\]
\end{proposition}

\vspace{.1in}

With the help of Proposition\ref{inq-tildeVgeqe}, we can now estimate the volume growth.

\begin{theorem}
	Let $\Sigma^n \subset \mathbb{R}^{n+1}$ be a complete properly immersed $\lambda$-translating soliton with $\lambda>0$ and $H_{\inf}-\lambda>0$. Then for any $x \in \Sigma$, there exists a constant $C$ such that
	\[
	\operatorname{Vol}(\Sigma \cap B_r(x))\geqslant \frac{\tilde{V}(1)}{e^{\big(1+\lambda(H_{\inf}-\lambda)\big)}}\left(\frac{1+\lambda(H_{\inf}-\lambda)}{\lambda(H_{\inf}-\lambda)}\right)e^{\lambda(H_{\inf}-\lambda)r}+C, 
	\]
    for all $r \geqslant 1$, where
        \begin{equation*}
        C=V(1)-\frac{\tilde{V}(1)}{e}\left(\frac{1+\lambda(H_{\inf}-\lambda)}{\lambda(H_{\inf}-\lambda)}\right).
    \end{equation*}
\end{theorem}

\begin{proof}
    For any $R > 1$, we have
\[
\tilde{V}'(R) \leqslant e^R \int_{\partial B_R \cap \Sigma} \frac{|X|}{|X^{\top}|} = e^R V'(R).
\]
Combining this with Proposition \ref{inq-tildeVgeqe} and \eqref{inq-tildeV'geqtildeV} gives that 
\begin{align*}
    C\bigr(1+\lambda(H_{\inf}-\lambda)\bigl)e^{\big(1+\lambda(H_{\inf}-\lambda)\big)R} 
   \leqslant & \bigr(1+\lambda(H_{\inf}-\lambda)\bigl)\tilde{V}(R) \\
   \leqslant & \tilde{V}'(R) \leqslant e^R V'(R),
\end{align*}
i.e.,
\begin{equation}\label{inq-eV}
    C\bigr(1+\lambda(H_{\inf}-\lambda)\bigl)e^{\lambda(H_{\inf}-\lambda)R} \leqslant V'(R).
\end{equation}
Integrating \eqref{inq-eV} from $1$ to $r$ yields
\[
V(r) - V(1) \geqslant C\left(\frac{1+\lambda(H_{\inf}-\lambda)}{\lambda(H_{\inf}-\lambda)}\right)\left(e^{\lambda(H_{\inf}-\lambda)r}-e^{\lambda(H_{\inf}-\lambda)}\right).
\]
i.e.,
\begin{equation*}
    \begin{split}
        V(r) &\geqslant C\left(\frac{1+\lambda(H_{\inf}-\lambda)}{\lambda(H_{\inf}-\lambda)}\right)\left(e^{\lambda(H_{\inf}-\lambda)r}-e^{\lambda(H_{\inf}-\lambda)}\right)+V(1)\\
        &=\frac{\tilde{V}(1)}{e^{\big(1+\lambda(H_{\inf}-\lambda)\big)}}\left(\frac{1+\lambda(H_{\inf}-\lambda)}{\lambda(H_{\inf}-\lambda)}\right)e^{\lambda(H_{\inf}-\lambda)r}\\
        &\quad +V(1)-\frac{\tilde{V}(1)}{e}\left(\frac{1+\lambda(H_{\inf}-\lambda)}{\lambda(H_{\inf}-\lambda)}\right).
    \end{split}
\end{equation*}
\end{proof}

\vspace{.1in}

\section{Non-existence of \texorpdfstring{$\lambda$}--Translating Solitons}

\vspace{.1in}

In this section, we will prove two non-existence results for graphical $\lambda$-translating solitons with bounded gradient. We will first show that:

\begin{proposition}\label{Proposition-infH-lambda}
	Let \( \Sigma^{2} \) be an oriented complete noncompact properly immersed $\lambda$-translating soliton in \( \mathbb{R}^{3} \) with $\lambda\geqslant0$, $\langle \mathbf{T},\nu \rangle>0$ and quadratic area growth.  
	Then:
	\begin{equation*}
		\inf_{\Sigma}\big(H-\lambda\big)=\inf_{\Sigma}\langle \mathbf{T},\nu \rangle=0.
	\end{equation*}
\end{proposition}

\begin{proof}
	We prove it by contradiction. Suppose there is a $\lambda$-translating soliton \( \Sigma^{2} \) in \( \mathbb{R}^{3} \) with $\lambda\geqslant0$, $\langle \mathbf{T},\nu \rangle>0$,
    quadratic area growth and 
	\begin{equation*}
		\inf_{\Sigma}\big(H-\lambda\big)=\inf_{\Sigma}\langle \mathbf{T},\nu \rangle= \delta>0.
	\end{equation*}
    From (\ref{e-T}), we have 
    \begin{equation*}
        |{\bf V}|^2=|{\bf T}|^2-(H-\lambda)^2\leqslant 1-\delta^2.
    \end{equation*}
	Set $f=\frac{1}{\langle \mathbf{T}, \nu \rangle}$, then
	\begin{equation}\label{inqboundoff}
		0< f\leqslant\frac{1}{\delta}.
	\end{equation}
	By \eqref{eq-Deltaanu}, we have
	\begin{equation}\label{eq-f}
		\Delta f=|B|^2f+\frac{2}{f}|\nabla f|^2-\langle {\bf V}, \nabla f \rangle.
	\end{equation}
	
	Let \(\phi\) be any cutoff function and \(p\) be a positive number to be determined later. 
	Multiplying both sides of \eqref{eq-f} by \(\phi^2 f^p\) and integrating by parts yields
	\begin{equation*}
		\begin{split}
			&\int_{\Sigma} \phi^2 f^{p+1} |B|^2 \, d\mu 
			+ 2 \int_{\Sigma} \phi^2 f^{p-1} |\nabla f|^2 \, d\mu 
			- \int_{\Sigma} \phi^2 f^p \langle {\bf V}, \nabla f \rangle \, d\mu \\
			=& \int_{\Sigma} \phi^2 f^p \Delta f \\
			=& -p \int_{\Sigma} \phi^2 f^{p-1} |\nabla f|^2 \, d\mu 
			- 2 \int_{\Sigma} \phi f^p \langle \nabla \phi, \nabla f \rangle \, d\mu .
		\end{split}
	\end{equation*}
	Rearranging this equality and using Young’s inequality, we obtain
	\begin{equation*}
		\begin{split}
			&p\int_{\Sigma} \phi^2 f^{p-1} |\nabla f|^2 \, d\mu 
			+ 2\int_{\Sigma} \phi^2 f^{p-1} |\nabla f|^2 \, d\mu + \int_{\Sigma} \phi^2 f^{p+1} |B|^2 \, d\mu \\
			=& \int_{\Sigma} \phi^2 f^p \langle {\bf V}, \nabla f \rangle \, d\mu 
			- 2 \int_{\Sigma} \phi f^p \langle \nabla \phi, \nabla f \rangle \, d\mu \\
			\leqslant &\int_{\Sigma} \phi^2 f^p |{\bf V}| \, |\nabla f| \, d\mu 
			+ 2 \int_{\Sigma} \phi f^p |\nabla \phi| \, |\nabla f| \, d\mu \\
			\leqslant &\epsilon \int_{\Sigma} \phi^2 f^{p+1} |{\bf V}|^2 \, d\mu 
			+ \frac{1}{4\epsilon} \int_{\Sigma} \phi^2 f^{p-1} |\nabla f|^2 \, d\mu \\
			\quad& + \int_{\Sigma} \phi^2 f^{p-1} |\nabla f|^2 \, d\mu 
			+ \int_{\Sigma} f^{p+1} |\nabla \phi|^2 \, d\mu ,
		\end{split}
	\end{equation*}
	which implies that
	\begin{equation}\label{inqfphi1}
		\begin{split}
			&\left( p + 1 - \frac{1}{4\epsilon} \right) \int_{\Sigma} \phi^2 f^{p-1} |\nabla f|^2 \, d\mu 
			+ \int_{\Sigma} \phi^2 f^{p+1} (|B|^2 - \epsilon |{\bf V}|^2) \, d\mu \\
			\leqslant & \int_{\Sigma} f^{p+1} |\nabla \phi|^2 \, d\mu.
		\end{split}
	\end{equation}
	Furthermore,
	\begin{equation}\label{inqfphi2}
		\begin{split}
			& \int_{\Sigma} \phi^2 f^{p+1} (\frac{1}{2}|B|^2 + \frac{1}{4}\delta^2 - \epsilon (1-\delta^2)) \, d\mu \\
			\leqslant &  \int_{\Sigma} \phi^2 f^{p+1} (\frac{1}{2}|B|^2 + \frac{1}{4}|H|^2 - \epsilon |{\bf V}|^2) \, d\mu \\
			\leqslant &  \int_{\Sigma} \phi^2 f^{p+1} (|B|^2 - \epsilon |{\bf V}|^2) \, d\mu.
		\end{split}
	\end{equation}
	We first choose \( \epsilon = \frac{\delta^2}{4(1-\delta^2)} \) so that \( \frac{1}{4}\delta^2 - \epsilon (1-\delta^2) = 0 \), 
	then take \( p = \frac{1-\delta^2}{\delta^2} \) so that \( p + 1 - \frac{1}{4\epsilon} = 1 \). 
    Note that \(\delta = 1\) if and only if \(\langle T,\nu\rangle \equiv 1\); in this case \(\Sigma\) is a hyperplane perpendicular to \(T\) with \(\lambda = -1\), which has already been excluded by \(\lambda \geqslant 0\).
	Then we obtain from \eqref{inqfphi1} and \eqref{inqfphi2} that
	\begin{equation}\label{inqfphi3}
		\begin{split}
			&\int_{\Sigma} \phi^2 f^{p-1} |\nabla f|^2 \, d\mu 
			+ \frac{1}{2}\int_{\Sigma} \phi^2 f^{p+1}|B|^2 \, d\mu 
			\leqslant \int_{\Sigma} f^{p+1} |\nabla \phi|^2 \, d\mu.
		\end{split}
	\end{equation}
	Next, we will choose appropriate cutoff function to deduce that \( f \) is a constant function and $|B|^2\equiv0$. 
	We will use the logarithmic cutoff argument. Let \( R > 1 \) be any fixed number. 
	Define the cutoff function \( \phi \) on \( \mathbb{R}^3 \) as follows: Let \( r \) denote the distance to the origin in \( \mathbb{R}^3 \), define
	\[
	\phi = 
	\begin{cases} 
		1, & r^2 \leqslant R; \\
		2 - 2 \frac{\log r}{\log R}, & R < r^2 \leqslant R^2; \\
		0, & r^2 > R^2.
	\end{cases}
	\]
	Recall that we assume the $\lambda$-translating soliton has quadratic area growth. 
	From \eqref{inqboundoff} and \eqref{inqfphi3}, we have
	\begin{equation*}
		\begin{split}
			&\int_{\Sigma \cap B(0, \sqrt{R})} f^{p-1} |\nabla f|^2 \, d\mu
			+ \frac{1}{2}\int_{\Sigma \cap B(0, \sqrt{R})} f^{p+1}|B|^2 \, d\mu  \\
			\leqslant& \int_{\Sigma} \phi^2 f^{p-1} |\nabla f|^2 \, d\mu 
			+ \frac{1}{2}\int_{\Sigma} \phi^2 f^{p+1}|B|^2 \, d\mu \\
			\leqslant& \int_\Sigma f^{p+1} |\nabla \phi|^2 \, d\mu \\
			\leqslant& \frac{4}{\delta^{p+1}(\log R)^2} \sum_{\frac{1}{2}\log R \leqslant l \leqslant \log R} \int_{\Sigma \cap (B(0, e^l) \setminus B(0, e^{l-1}))} r^{-2} \, d\mu \\
			\leqslant& \frac{4}{\delta^{p+1}(\log R)^2} \sum_{\frac{1}{2}\log R \leqslant l \leqslant \log R} e^{-2(l-1)} D_0 e^{2l} \\
			\leqslant& \frac{C}{\log R},
		\end{split}
	\end{equation*}
	where \( C \) depends only on \( D_0 \) and \( \delta \). 
	As \( f>0 \), letting \( R \to \infty \), we get that \( f \) is a constant and $|B|^2\equiv0$.
	Next, based on the inequality
	\begin{equation*}
		|B|^2\geqslant\frac{1}{2}|\mathbf{H}|^2,
	\end{equation*}
	we know $H=0$ and $\inf_{\Sigma}\big(H-\lambda\big)\leqslant0.$
	This gives the desired contradiction.
\end{proof}

Using Proposition \ref{Proposition-infH-lambda}, we can prove that

\begin{theorem}
	When $n = 2$ and $\lambda\geqslant0$, there is no entire solution to the equation \eqref{eq-graphlambdatranslatingsoliton}
	with bounded gradient.
\end{theorem}

\begin{proof}
	We prove it by contradiction. Suppose there is an entire solution $u$ to the equation \eqref{eq-graphlambdatranslatingsoliton} defined on the whole $\mathbb{R}^2$ with
	\begin{equation*}\label{eq-DuM}
		|Du| \leqslant M.
	\end{equation*}
	From \eqref{eq-graphlambdatranslatingsoliton}, we see that
	\begin{equation*}\label{eq:4.9}
		H-\lambda = \frac{1}{\sqrt{1+|Du|^2}}\geqslant \frac{1}{\sqrt{1+M^2}}.
	\end{equation*}
	Next, we will show that $\Sigma$ has quadratic area growth. We denote by $\hat{B}(0, r)$ the ball of radius $r$ centered at $0$ in the domain plane $\mathbb{R}^2$, while denote by $B(0, r)$ the ball of radius $r$ centered at $0$ in $\mathbb{R}^3$. It is obvious that
	\[
	\Sigma \cap B(0, r) \subset \mathrm{Graph}_u(\hat{B}(0, r)).
	\]
	Therefore,
	\begin{equation*}\label{eq:4.10}
		\begin{aligned}
			\mathrm{Area}(\Sigma \cap B(0, r)) &\leqslant \mathrm{Area}(\mathrm{Graph}_u(\hat{B}(0, r))) \\
			&= \int_{\hat{B}(0,r)} \sqrt{1+|Du|^2} dxdy \\
			&\leqslant \sqrt{1+M^2}\pi r^2 \equiv C_1 r^2.
		\end{aligned}
	\end{equation*}
	Combining the above together, we find a proper complete $\lambda$-translating soliton in $\mathbb{R}^3$ with $H-\lambda\geqslant \frac{1}{\sqrt{1+M^2}} > 0$ and quadratic area growth. 
    This contradicts the Proposition \ref{Proposition-infH-lambda}.
\end{proof}

\vspace{.1in}

For high dimensional case, we only consider the situation that $\Sigma$ is rotationally symmetric. In this case, the graph function $u$ satisfies the equation (\ref{eq-rotationalgraphlambdatranslatingsoliton}). We first have the following observation:

\begin{proposition}\label{Proposition-infH-lambda2}
Let  $u=u(r)\in C^{2}[0,+\infty)$ be a single variable function with bounded derivative satisfying (\ref{eq-rotationalgraphlambdatranslatingsoliton}). Then there exists a sequence $r_i \to \infty$, such that
\[
\lim_{i\to\infty} u''(r_i) = 0.
\]
\end{proposition}
\begin{proof}
Denote
\[
\liminf_{r\to\infty} u''(r) = \xi \quad\text{and} \quad\limsup_{r\to\infty} u''(r) = \eta,
\]
where we allow $\xi,\eta\in[-\infty,+\infty]$.
Note equation (\ref{eq-rotationalgraphlambdatranslatingsoliton}) together with $|u'| \leqslant M$ yields 
$|u''(r)| \leqslant (1+M^{2})(1 + \lambda(1+M^{2})^{\frac{1}{2}} + \frac{(n-1)M}{r})$, so $\xi$ and $\eta$ are in fact finite. 
The following case argument, however, is valid whether or not they are finite. 

\noindent\textbf{Case 1: $\xi < 0$, $\eta > 0$.}
In this case, there will be two sequences $a_i \to \infty$ and $b_i \to \infty$, such that
\[
u''(a_i) < 0,\quad u''(b_i) > 0.
\]
By selecting subsequences if necessary, we may assume that for each $i$
\[
a_i < b_i < a_{i+1} < b_{i+1}.
\]
Since $u''$ is continuous, we see that there exists $r_i \in (a_i, b_i)$, such that $u''(r_i) = 0$.
It is obvious that $r_i \to \infty$.

\noindent\textbf{Case 2: $\xi = 0$ or $\eta = 0$.}
In this case, it is obvious that there will be a sequence $r_i \to \infty$, such that
\[
\lim_{i\to\infty} u''(r_i) = 0.
\]

\noindent\textbf{Case 3: $\xi > 0$ or $\eta < 0$.}
Without loss of generality, we assume that $\xi > 0$.
Then there exists a constant $K > 0$, such that $u''(r) \geqslant \dfrac{\xi}{2} > 0$ for all $r \geqslant K$.
Therefore, for $r > K$,
\[
u'(r) \geqslant u'(K) + \dfrac{\xi}{2}(r - K) \to \infty \quad \text{as } r \to \infty,
\]
which contradicts bounded gradient. 

\end{proof}

Using Proposition \ref{Proposition-infH-lambda2}, we can prove that

\begin{theorem}
    There does not exist a rotationally symmetric graphical complete $\lambda$-translating soliton
    with bounded gradient for $\lambda\geqslant0$.
\end{theorem}
\begin{proof}
    We prove it by contradiction. Suppose there is a rotational entire solution $u$ to the equation \eqref{eq-rotationalgraphlambdatranslatingsoliton} defined on the whole $\mathbb{R}^n$ with
	\begin{equation*}\label{eq-DuM1}
		|u'| \leqslant M.
	\end{equation*}
    Then we have from \eqref{eq-rotationalgraphlambdatranslatingsoliton} that
    \begin{equation*}
    \frac{u''}{\bigr(1+(u')^{2} \bigl)^{\frac{3}{2}}}+\frac{(n-1)u'}{r\sqrt{1+(u')^2}}=\frac{1}{\sqrt{1+(u')^{2} }}+\lambda
    \geqslant \frac{1}{\sqrt{1+M^{2} }}+\lambda
    >0.
\end{equation*}
Now we can take values at $r_i$ given by Proposition \ref{Proposition-infH-lambda2} in the above inequality and let $r_i\to\infty $ to obtain the desired contradiction.
\end{proof}

\vspace{.2in}

\section{Rigidity Results for \texorpdfstring{$\lambda$}--Translating Solitons}

\vspace{.1in}

In this section, we will prove two rigidity theorems for $\lambda$-Translating Solitons.

\subsection*{Rigidity Theorem in terms of Gauss Map of \texorpdfstring{$\lambda$}--Translating Solitons}

In this subsection, we first recall the following test function considered by Bao-Shi (\cite{bao2014gauss}):

\begin{lemma}[{\cite{bao2014gauss}}]\label{lem:test-function}
	On any ball $B_\Lambda^{S^n}(y_0)$ of $\mathbb{S}^n$, $\Lambda < \frac{\pi}{2}$, let $\overline{d}$ be the distance function from $y_0$ on $\mathbb{S}^n$, we define $\varphi(y) = 1 - \cos \overline{d}(y)$ on $B_\Lambda^{S^n}(y_0)$, then $\varphi$ satisfies the following properties:
	\begin{enumerate}[(1)]
		\item There exists a constant $b$, such that $0 \leqslant \varphi < b < 1$;
		\item $\frac{d\varphi}{d \overline{d}} = \sin \overline{d}$;
		\item $\operatorname{Hess}\varphi = (\cos \overline{d})I$, where $\operatorname{Hess}\varphi$ is the hessian of $\varphi$, and $I$ is the identity matrix.
	\end{enumerate}
\end{lemma}

With this lemma and lemma \ref{lem:bochner}, we can obtain the rigidity theorem in terms of Gauss map of $\lambda$-translating solitons.

\begin{theorem}\label{thm:main1}
	Let $\Sigma^n \subset \mathbb{R}^{n+1}$ be an $n$-dimensional complete $\lambda$-translating soliton with $\lambda B$ positive semi-definite. If the image of Gauss map $\gamma$ of $\Sigma^n$ lies in a ball $B_\Lambda^{S^n}(y_0)$ of $\mathbb{S}^n$, where $\Lambda < \frac{\pi}{2}$, then $\Sigma^n$ must be a hyperplane.
\end{theorem}

\begin{proof}
	Choosing a convex function on $B_\Lambda^{S^n}(y_0)$ as above, by direct computation we have
	\[
	\Delta\varphi(\gamma(x)) = \operatorname{Hess}\varphi(\nabla \gamma, \nabla \gamma) + \langle D\varphi, \tau(\gamma) \rangle = \cos \overline{d}|\nabla \gamma|^2 - \langle \mathbf{V}, \nabla\varphi \rangle.
	\]
	Define $\phi(x) = \frac{|\nabla \gamma|^2(x)}{(b - \varphi(\gamma(x)))^2}$. Then
	\[
	\nabla\phi(x) = \frac{\nabla|\nabla \gamma|^2}{(b - \varphi)^2} + \frac{2|\nabla \gamma|^2\nabla\varphi}{(b - \varphi)^3}.
	\]
	and
	\begin{equation}\label{eq-Delatphi}
		\begin{split}
			\Delta\phi(x) &= \frac{\Delta|\nabla \gamma|^2}{(b-\varphi)^2} + \frac{4\langle\nabla\varphi,\nabla|\nabla \gamma|^2\rangle}{(b-\varphi)^3} + \frac{2\Delta\varphi|\nabla \gamma|^2}{(b-\varphi)^3} + \frac{6|\nabla\varphi|^2|\nabla \gamma|^2}{(b-\varphi)^4}\\
			&= \frac{2|\nabla d\gamma|^2 - \langle \mathbf{V},\nabla|\nabla \gamma|^2\rangle - 2|\nabla \gamma|^4}{(b-\varphi)^2} + \frac{4\langle\nabla\varphi,\nabla|\nabla \gamma|^2\rangle}{(b-\varphi)^3} + \frac{2\lambda B(\nabla \gamma,\nabla \gamma)}{(b-\varphi)^2} \\
			&\quad + \frac{2\cos \overline{d}|\nabla \gamma|^4 - 2\langle \mathbf{V},\nabla\varphi\rangle|\nabla \gamma|^2}{(b-\varphi)^3} + \frac{6|\nabla\varphi|^2|\nabla \gamma|^2}{(b-\varphi)^4},
		\end{split}
	\end{equation}
	because
	\[
	\langle \mathbf{V},\nabla\phi\rangle = \frac{\langle \mathbf{V},\nabla|\nabla \gamma|^2\rangle}{(b-\varphi)^2} + \frac{2|\nabla \gamma|^2\langle \mathbf{V},\nabla\varphi\rangle}{(b-\varphi)^3},
	\]
	\[
	\frac{\langle\nabla\varphi,\nabla\phi\rangle}{b-\varphi} = \frac{\langle\nabla\varphi,\nabla|\nabla \gamma|^2\rangle}{(b-\varphi)^3} + \frac{2|\nabla \gamma|^2|\nabla\varphi|^2}{(b-\varphi)^4},
	\]
	and
	\[
	\frac{2|\nabla d\gamma|^2}{(b-\varphi)^2} + \frac{2|\nabla\varphi|^2|\nabla \gamma|^2}{(b-\varphi)^4} \geqslant \frac{4|\nabla\varphi||\nabla \gamma||\nabla d\gamma|}{(b-\varphi)^3}.
	\]
	Then \eqref{eq-Delatphi} becomes
	\[
	\begin{aligned}
		\Delta\phi(x) &\geqslant \frac{2\cos \overline{d}|\nabla \gamma|^4}{(b-\varphi)^3} - \frac{2|\nabla \gamma|^4}{(b-\varphi)^2} + \frac{2\langle\nabla\varphi,\nabla\phi\rangle}{b-\varphi} - \langle \mathbf{V},\nabla\phi\rangle + \frac{2\lambda B(\nabla \gamma,\nabla \gamma)}{(b-\varphi)^2} \\
	    &\geqslant 2\cos \overline{d}(b-\varphi)\phi^2 - 2(b-\varphi)^2\phi^2 + \frac{2\langle\nabla\varphi,\nabla\phi\rangle}{b-\varphi} - \langle \mathbf{V},\nabla\phi\rangle + \frac{2\lambda B(\nabla \gamma,\nabla \gamma)}{(b-\varphi)^2} \\
		&\geqslant 2\cos \overline{d}(b-\varphi)\phi^2 - 2(b-\varphi)^2\phi^2 + \frac{2\langle\nabla\varphi,\nabla\phi\rangle}{b-\varphi} - \langle \mathbf{V},\nabla\phi\rangle,
	\end{aligned}
	\]
	where we used the assumption that $\lambda B$ is positive semi-definite.
    
    Once we have obtained the above inequality, the remaining step is the
    localization argument of Bao--Shi (\cite{bao2014gauss}), and we only
    indicate why it applies here without any extra hypothesis. Their bounded mean curvature assumption enters only through the estimate
    $\Delta r^{2}=2n+2\langle\mathbf{H},X\rangle\leqslant C(1+r)$ for the extrinsic distance $r(x)=|X(x)|$; in our setting this is automatic, since the soliton equation \eqref{eq-lambdatranslatingsoliton} and $|\mathbf{T}|=1$ give
    $|H|=|\langle \mathbf{T},\nu\rangle+\lambda|\leqslant1+|\lambda|$. 
    
    The only other input of their argument is the boundedness of the drift vector, which in our case follows from $|V|\leqslant|\mathbf{T}|=1$. Since $\Sigma$ is complete, $\Sigma\cap\overline{B^{n+1}_R(0)}$ is compact, so $G=(R^{2}-r^{2})^{2}\phi$ attains its maximum at an interior point $x_{0}$, where $\nabla G=0$ and $\Delta G\leqslant0$.
    Substituting $\nabla\phi/\phi=4r\nabla r/(R^{2}-r^{2})$ into the above
    inequality, multiplying by $(R^{2}-r^{2})^{2}$, and using
    $\cos\overline{d}-(b-\varphi)=1-b>0$ and
    $b-\varphi\geqslant b-(1-\cos\Lambda)>0$, one obtains
    $G(x_{0})^{\frac12}\leqslant C\bigl(R^{\frac32}+R\bigr)$ with $C$ independent of $R$. Since $R^{2}-r^{2}\geqslant\frac34R^{2}$ on $\Sigma\cap B^{n+1}_{R/2}(0)$, this yields
    \[
     \sup_{\Sigma\cap B^{n+1}_{R/2}(0)}
     \frac{|\nabla\gamma|}{b-\varphi(\gamma)}
     \leqslant C\Bigl(R^{-\frac12}+R^{-1}\Bigr)\longrightarrow0
     \quad\text{as }R\to\infty.
    \]
    Hence $\nabla\gamma\equiv0$ on $\Sigma$, that is, $h_{ij}\equiv0$, and
    $\Sigma$ must be a hyperplane.
\end{proof}

\subsection*{Rigidity Theorem under Integral Assumptions}

In order to get rigidity result for $\lambda$-translating solitons under the integral assumption, we need the following lemma. 

\begin{lemma}[{\cite{michael1973sobolev}}]\label{lem:sobolev}
	Let $\mathcal{M}$ be a smooth immersed submanifold in $\mathbb{R}^{n+1}$ and $g$ be a non-negative smooth function with compact support, then
	\begin{equation}\label{eq:sobolev}
		\kappa^{-1}\left(\int_\mathcal{M} g^{\frac{2n}{n-2}}d\mu\right)^{\frac{n-2}{n}} \leqslant \int_\mathcal{M} |\nabla g|^2 d\mu + \frac{1}{2}\int_\mathcal{M} |\mathbf{H}|^2 g^2 d\mu,
	\end{equation}
	where $\kappa$ is a constant.
\end{lemma}

Now we can prove the rigidity theorem under integral assumption:

\begin{theorem}\label{thm:main2}
	Let $\Sigma^n$ be a complete immersed $\lambda$-translating soliton in $\mathbb{R}^{n+1}$. If $ \lambda H \leqslant 0$, and $\Sigma^n$ satisfies integral conditions
	\[
	\left(\int_{\Sigma^n} |B|^n\right)^{\frac{2}{n}} + |\lambda|\left(\int_{\Sigma^n} |B|^{\frac{n}{2}}\right)^{\frac{2}{n}} < \frac{2(n-1)}{n^2\kappa} \quad \text{and} \quad \int_{\Sigma^n} |B|^n e^{\langle \mathbf{T},X \rangle} < \infty,
	\]
	where $\kappa$ is the Sobolev constant, then $|B| \equiv 0$ and $\Sigma^n$ is a hyperplane.
\end{theorem}

\begin{proof}
	Let $\eta$ be a smooth function with compact support in $\Sigma^n$. Multiplying $\eta^2|B|^{n-2}\rho$ on both sides of \eqref{eq:L-A-squared} yields
	\[
	\begin{aligned}
		\int_{\Sigma^n} \eta^2|B|^{n-2}\rho\mathcal{L}|B|^2 &= 2\int_{\Sigma^n} |\nabla B|^2|B|^{n-2}\eta^2\rho - 2\int_{\Sigma^n} |B|^{n+2}\eta^2\rho \\
		&\quad + 2\lambda\int_{\Sigma^n} |B|^{n-2}\eta^2\operatorname{tr}(B^3)\rho.
	\end{aligned}
	\]
	On the other hand, integrating by parts yields
	\[
	\begin{aligned}
		\int_{\Sigma^n} \eta^2|B|^{n-2}\rho\mathcal{L}|B|^2 &= -\int_{\Sigma^n} \langle \nabla(\eta^2|B|^{n-2}),\nabla|B|^2 \rangle\rho \\
		&= -\int_{\Sigma^n} 2(n-2)\eta^2|B|^{n-2}|\nabla |B||^2\rho - 4\int_{\Sigma^n}|B|^{n-1}\eta\langle\nabla\eta,\nabla|B|\rangle\rho.
	\end{aligned}
	\]
	Using Kato's inequality, we have
	\[
	\begin{aligned}
		0 &\geqslant (n-1)\int_{\Sigma^n} \eta^2|B|^{n-2}|\nabla |B||^2\rho + 2\int_{\Sigma^n} |B|^{n-1}\eta\langle\nabla\eta,\nabla|B|\rangle\rho \\
		&\quad - \int_{\Sigma^n} |B|^{n+2}\eta^2\rho + \lambda\int_{\Sigma^n} |B|^{n-2}\eta^2\operatorname{tr}(B^3)\rho.
	\end{aligned}
	\]
	By the Cauchy inequality, for any $\varepsilon > 0$, the above inequality becomes
	\begin{equation}\label{eq:cauchy-ineq}
		\begin{aligned}
			(n-1)\int_{\Sigma^n} \eta^2|B|^{n-2}|\nabla |B||^2\rho &\leqslant 2\int_{\Sigma^n} |B|^{n-1}\eta|\nabla\eta||\nabla |B||\rho \\
			&\quad + \int_{\Sigma^n} |B|^{n+2}\eta^2\rho - \lambda\int_{\Sigma^n} |B|^{n-2}\eta^2\operatorname{tr}(B^3)\rho \\
			&\leqslant \varepsilon\int_{\Sigma^n} |B|^{n-2}\eta^2|\nabla |B||^2\rho + \frac{1}{\varepsilon}\int_{\Sigma^n} |B|^n|\nabla\eta|^2\rho \\
			&\quad + \int_{\Sigma^n} |B|^{n+2}\eta^2\rho - \lambda\int_{\Sigma^n} |B|^{n-2}\eta^2\operatorname{tr}(B^3)\rho.
		\end{aligned}
	\end{equation}
	
	Let $f = |B|^{\frac{n}{2}}\rho^{\frac{1}{2}}\eta$. 
	Integrating by parts and using \eqref{eq:grad-rho}, \eqref{eq:grad-rho-half} and \eqref{eq:laplace-rho} then we have
	\begin{equation}\label{eq:grad-f}
		\begin{split}
			&\quad\int_{\Sigma^n} |\nabla f|^2 \\
            &= \int_{\Sigma^n} |\nabla(|B|^{\frac{n}{2}}\eta)|^2\rho + 2\int_{\Sigma^n} |B|^{\frac{n}{2}}\eta\nabla(|B|^{\frac{n}{2}}\eta)\rho^{\frac{1}{2}}\nabla\rho^{\frac{1}{2}} + \int_{\Sigma^n} |B|^n\eta^2|\nabla\rho^{\frac{1}{2}}|^2 \\
			&= \int_{\Sigma^n} |\nabla(|B|^{\frac{n}{2}}\eta)|^2\rho + \frac{1}{2}\int_{\Sigma^n} \nabla(|B|^n\eta^2)\nabla\rho + \frac{1}{4}\int_{\Sigma^n} |B|^n\eta^2|\mathbf{V}|^2\rho \\
			&= \int_{\Sigma^n} |\nabla(|B|^{\frac{n}{2}}\eta)|^2\rho - \frac{1}{2}\int_{\Sigma^n} |B|^n\eta^2\Delta\rho + \frac{1}{4}\int_{\Sigma^n} |B|^n\eta^2|\mathbf{V}|^2\rho \\
			&= \int_{\Sigma^n} |\nabla(|B|^{\frac{n}{2}}\eta)|^2\rho - \frac{1}{2}\int_{\Sigma^n} |B|^n\eta^2\rho\big(1 + \lambda(H - \lambda)\big) + \frac{1}{4}\int_{\Sigma^n} |B|^n\eta^2|\mathbf{V}|^2\rho,
		\end{split}
	\end{equation}
	Combining \eqref{eq-lambdatranslatingsoliton}, the Sobolev inequality \eqref{eq:sobolev} and \eqref{eq:grad-f}, we have
	\begin{equation}\label{eq:sobolev-combine}
		\begin{split}
			\kappa^{-1}\left(\int_{\Sigma^n} |f|^{\frac{2n}{n-2}}\right)^{\frac{n-2}{n}} &\leqslant \int_{\Sigma^n} |\nabla f|^2 + \frac{1}{2}\int_{\Sigma^n} |B|^n\eta^2 H^2\rho \\
			&\leqslant \int_{\Sigma^n} |\nabla(|B|^{\frac{n}{2}}\eta)|^2\rho - \frac{1}{2}\int_{\Sigma^n} |B|^n\eta^2\rho\big(1 + \lambda(H - \lambda)\big) \\
			&\quad + \frac{1}{4}\int_{\Sigma^n} |B|^n\eta^2|\mathbf{V}|^2\rho + \frac{1}{2}\int_{\Sigma^n} |B|^n\eta^2 H^2\rho \\
			&= \frac{1}{4}\int_{\Sigma^n} |B|^n\eta^2\rho\big[|\mathbf{V}|^2 - 2\big(1 + \lambda(H - \lambda)\big) + 2H^2\big]\\
            &\quad+ \int_{\Sigma^n} |\nabla(|B|^{\frac{n}{2}}\eta)|^2\rho.
		\end{split}
	\end{equation}
	Combining the Cauchy inequality, \eqref{eq:cauchy-ineq} and \eqref{eq:sobolev-combine}, we have

\begin{equation}\label{eq:holder-combine}
		\begin{split}
			&\quad\kappa^{-1}\left(\int_{\Sigma^n} |f|^{\frac{2n}{n-2}}\right)^{\frac{n-2}{n}}\\
			&\leqslant \frac{n^2}{2}\int_{\Sigma^n} |\nabla |B||^2|B|^{n-2}\eta^2\rho + 2\int_{\Sigma^n} |B|^n|\nabla\eta|^2\rho - \frac{1}{4}\int_{\Sigma^n} |B|^n\eta^2\rho|\mathbf{V}|^2 \\
            &\quad +\frac{1}{2}\int_{\Sigma^n} |B|^n\eta^2\rho\lambda H
            \\
			&\leqslant \frac{n^2}{2(n-1-\varepsilon)}\left(\int_{\Sigma^n} |B|^{n+2}\eta^2\rho + \frac{1}{\varepsilon}\int_{\Sigma^n} |B|^n|\nabla\eta|^2\rho\right) \\
			&\quad + 2\int_{\Sigma^n} |B|^n|\nabla\eta|^2\rho - \frac{1}{4}\int_{\Sigma^n} |B|^n\eta^2\rho|\mathbf{V}|^2 +\frac{1}{2}\int_{\Sigma^n} |B|^n\eta^2\rho\lambda H
            \\
            &\quad -\frac{\lambda n^2}{2(n-1-\varepsilon)}\int_{\Sigma^n} |B|^{n-2}\eta^2\operatorname{tr}(B^3)\rho\\
			&\leqslant \frac{n^2}{2(n-1-\varepsilon)}\left(\int_{\Sigma^n} |B|^{n+2}\eta^2\rho- \lambda\int_{\Sigma^n} |B|^{n-2}\eta^2\operatorname{tr}(B^3)\rho\right) \\
            &\quad+ \left(\frac{n^2}{2\varepsilon(n-1-\varepsilon)} + 2\right)\int_{\Sigma^n} |B|^n|\nabla\eta|^2\rho\\
            &\quad - \frac{1}{4}\int_{\Sigma^n} |B|^n\eta^2\rho|\mathbf{V}|^2 +\frac{1}{2}\int_{\Sigma^n} |B|^n\eta^2\rho\lambda H.
		\end{split}
	\end{equation}

    Notice that we have the inequality
    \[
    |tr(B^3)|\leqslant|B|^3.
    \]
	Using the Hölder inequality
	\[
	\int_{\Sigma^n} |B|^{n+2}\eta^2\rho \leqslant \left(\int_{\Sigma^n} |B|^n\right)^{\frac{2}{n}} \left(\int_{\Sigma^n} (|B|^n\eta^2\rho)^{\frac{n}{n-2}}\right)^{\frac{n-2}{n}},
	\]
	and
	\[
	\int_{\Sigma^n} |B|^{n+1}\eta^2\rho \leqslant \left(\int_{\Sigma^n} |B|^{\frac{n}{2}}\right)^{\frac{2}{n}} \left(\int_{\Sigma^n} (|B|^n\eta^2\rho)^{\frac{n}{n-2}}\right)^{\frac{n-2}{n}}
	\]
	in \eqref{eq:holder-combine} yields
	\[
	\begin{aligned}
		&\quad\kappa^{-1}\left(\int_{\Sigma^n} |f|^{\frac{2n}{n-2}}\right)^{\frac{n-2}{n}}\\
		&\leqslant \left[\frac{n^2}{2(n-1-\varepsilon)}\left(\int_{\Sigma^n} |B|^n\right)^{\frac{2}{n}} + \frac{|\lambda|n^2}{2(n-1-\varepsilon)}\left(\int_{\Sigma^n} |B|^{\frac{n}{2}}\right)^{\frac{2}{n}}\right]\left(\int_{\Sigma^n} (|B|^n\eta^2\rho)^{\frac{n}{n-2}}\right)^{\frac{n-2}{n}}\\
        &\quad+ \left(\frac{n^2}{2\varepsilon(n-1-\varepsilon)} + 2\right)\int_{\Sigma^n} |B|^n|\nabla\eta|^2\rho- \frac{1}{4}\int_{\Sigma^n} |B|^n\eta^2\rho|\mathbf{V}|^2 +\frac{1}{2}\int_{\Sigma^n} |B|^n\eta^2\rho\lambda H\\
		&= \frac{n^2}{2(n-1-\varepsilon)}\left[\left(\int_{\Sigma^n} |B|^n\right)^{\frac{2}{n}} + |\lambda|\left(\int_{\Sigma^n} |B|^{\frac{n}{2}}\right)^{\frac{2}{n}}\right]\left(\int_{\Sigma^n} f^{\frac{2n}{n-2}}\right)^{\frac{n-2}{n}} \\
		&\quad + \left(\frac{n^2}{2\varepsilon(n-1-\varepsilon)} + 2\right)\int_{\Sigma^n} |B|^n|\nabla\eta|^2\rho - \frac{1}{4}\int_{\Sigma^n} f^{2}|\mathbf{V}|^2 +\frac{1}{2}\int_{\Sigma^n} f^{2}\lambda H.
	\end{aligned}
	\]
	If we assume that
	\[
	\left(\int_{\Sigma^n} |B|^n\right)^{\frac{2}{n}} + |\lambda|\left(\int_{\Sigma^n} |B|^{\frac{n}{2}}\right)^{\frac{2}{n}} < \frac{2(n-1)}{n^2\kappa} \quad \text{and} \quad \lambda H \leqslant 0,
	\]
	then we arrive at by choosing $\varepsilon>0$ sufficiently small that 
	\[
	\left(\int_{\Sigma^n} |f|^{\frac{2n}{n-2}}\right)^{\frac{n-2}{n}} \leqslant C(n,\varepsilon)\int_{\Sigma^n} |B|^n|\nabla\eta|^2\rho,
	\]
	where $C(n,\varepsilon)$ is a positive constant. Choosing an appropriate cut off function $\eta$ leads to the theorem.
\end{proof}

\end{document}